\documentclass[a4paper, 12pt, leqno]{article}

\usepackage{amssymb}
\usepackage{mathrsfs}
\usepackage{dsfont}
\usepackage{stmaryrd} 
\usepackage{enumerate, xspace}
\usepackage{changebar}
\usepackage{hyperref}
\usepackage{pdfsync}

\usepackage{amsthm, graphicx}

\usepackage{ascmac}
\usepackage{amsmath}

\theoremstyle{plain}

\def\hsymb#1{\mbox{\strut\rlap{\smash{\Huge$#1$}}\quad}}

\newtheorem{theorem}{Theorem}[section]
\newtheorem{lemma}[theorem]{Lemma}
\newtheorem{corollary}[theorem]{Corollary}
\newtheorem{proposition}[theorem]{Proposition}
\newtheorem{definition}[theorem]{Definition}
\newtheorem{remark}[theorem]{Remark}
\newtheorem{example}[theorem]{Example}

\numberwithin{equation}{section}

\makeatletter
  \def\varddots{\mathinner{\mkern1mu
      \raise\p@\hbox{.}\mkern2mu\raise4\p@\hbox{.}\mkern2mu
      \raise7\p@\vbox{\kern7\p@\hbox{.}}\mkern1mu}}
  \makeatother

\renewcommand{\a}{\alpha}

\newcommand{\la}{\lambda}

\def\R{{\mathbb{R}}}
\def\N{{\mathbb{N}}}

\def\C{{\mathbb{C}}}

\def\1{{\mathbf{1}}}
\allowdisplaybreaks

\begin{document}
\title
{Eigenvalue problem for some special class of anti-triangular matrices}
\author{Hiroyuki Ochiai, Makiko Sasada, Tomoyuki Shirai \\
and Takashi Tsuboi}

\maketitle

\begin{abstract}
We study the eigenvalue problem for some special class of anti-triangular matrices. Though the eigenvalue problem is quite classical, as far as we know, almost nothing is known about properties of eigenvalues for anti-triangular matrices. In this paper, we show that there is a nice class of anti-triangular matrices whose eigenvalues are given explicitly by their elements. Moreover, this class contains several interesting subclasses which we characterize in terms of probability measures. We also discuss the application of our main theorem to the study of interacting particle systems, which are stochastic processes studied in extensive literature. 
\end{abstract}



\section{Introduction}

The aim of this paper is to study the eigenvalue problem for anti-triangular matrices. In general, it is impossible to express the eigenvalues of a given matrix explicitly in terms of its elements. Obviously, for a triangular matrix, the eigenvalues are given by its diagonal elements. On the other hand, for an anti-triangular matrix, such a simple relation seems hard to expect. However, in this paper, we show that for a certain class of anti-triangular matrices, similar property holds. Moreover, this class contains several interesting subclasses which we characterize in terms of probability measures.

The eigenvalue problem is one of the most classical and well-studied topics in wide fields of mathematics. In the theory of Markov chains, an estimate of the eigenvalues of a transition matrix is essential to know the speed of convergence to an invariant measure (or stationary distribution). Recently, in the study of interacting particle systems, which are continuous time Markov processes with discrete state space, it is shown that an estimate of the eigenvalues of some special matrix obtained from an invariant measure is also useful to know the speed of convergence to the invariant measure. This matrix is, by the way of construction, an anti-triangular, stochastic and \lq\lq symmetric" (in the sense defined in this paper) matrix. Our original motivation was to estimate the eigenvalue of this type of matrix. Therefore, we also discuss what is shown for this original case by our main theorem in the last part of the paper.       

\section{Notations and main results}

\subsection{General case}


We denote by $M(n+1,\C)$ the $\C$-linear space of square matrices of size $n+1$ and by $M_{L}(n+1,\C)$ the $\C$-linear space of lower triangular matrices of size $n+1$.
Let $G=(g_{ij})_{i,j=0}^n \in M(n+1, \C)$ be the matrix with nonzero elements $g_{i,n-i}=1, i=0,1,\dots, n$; 
\begin{equation}\label{upper}
G = \left(
\begin{array}{ccccc}
&&&& 1  \\
 \hsymb{0}   && &   1 \\
&&   1 &  & \\
& \varddots &&& \\
1 &&&& \hsymb{0}
\end{array}
\right).
\end{equation}
Note that $G^{-1}=G$.

\begin{definition}
A matrix $X =(x_{ij})_{i,j=0}^n \in M (n+1,\C)$ is called a \lq\lq lower anti-triangular matrix" if $x_{ij}=0$ for $i+j < n$ ;

\begin{equation}\label{upper}
X = \left(
\begin{array}{ccccc}
&&&& x_{0,n}  \\
 \hsymb{0}   &&& \ x_{1,n-1} &  \\
&&  \varddots &&\\
& \varddots &&& \\
x_{n,0}&&&& \hsymb{*}
\end{array}
\right)
\end{equation}
In other words, a matrix $X =(x_{ij})_{i,j=0}^n \in M (n+1,\C)$ is \lq\lq lower anti-triangular" if and only if $X= \tilde{X}G$ for some $\tilde{X} \in M_L(n+1,\C)$. 
\end{definition}
Analogously, we can define upper anti-triangular matrices. The eigenvalue problem for the upper anti-triangular matrices is equivalent to that for the lower anti-triangular matrices since for any upper anti-triangular matrix $X$, $G^{-1}XG$ is lower anti-triangular and vice versa. Therefore, from now on, we only consider the lower anti-triangular matrices and as long as it is not stated otherwise, \lq\lq anti-triangular" always means \lq\lq lower anti-triangular" and \lq\lq triangular" always means \lq\lq lower triangular".

Generally, it is not easy to express the eigenvalues of an anti-triangular matrix in terms of its elements. On the other hand, there are interesting examples for which we are able to do this.

\begin{example}\label{exm:TS}
Let $T=(t_{ij})_{i,j=0}^n$ and $S=(s_{ij})_{i,j=0}^n \in M_L(n+1,\C)$ be matrices with elements $t_{ij}=\frac{1}{i+1} \1_{\{i \ge j \}}$ and $s_{ij}=\binom{i}{j}\frac{1}{2^i}\1_{\{i \ge j \}}$. Then, $TG$ and $SG$ are anti-triangular matrices written as
\begin{equation}\label{upper}
T G = \left(
\begin{array}{ccccc}
&&&& 1  \\
 \hsymb{0}   &&& \ \frac{1}{2} & \frac{1}{2}  \\
&&  \frac{1}{3} & \frac{1}{3} & \frac{1}{3}  \\
& \varddots &&& \vdots  \\
\frac{1}{n+1} &   \frac{1}{n+1} & \dots &&  \frac{1}{n+1}
\end{array}
\right)
\end{equation}
and 
\begin{equation}\label{upper}
S G = \left(
\begin{array}{ccccc}
&&&& 1  \\
 \hsymb{0}   &&& \ \frac{1}{2} & \frac{1}{2}  \\
&&  \frac{1}{4} & \frac{1}{2} & \frac{1}{4}  \\
& \varddots &&& \vdots  \\
\frac{1}{2^n} &   \frac{n}{2^n} & \dots & \frac{\binom{n}{k}}{2^n}  \dots & \frac{1}{2^n}
\end{array}
\right).
\end{equation}




Remarkably, the eigenvalues of $TG$ are 
\[ 
\Big( 1,-\frac{1}{2},\frac{1}{3},-\frac{1}{4},\dots, (-1)^n\frac{1}{n+1} \Big) = \Big(  t_{00},- t_{11}, t_{22},-t_{33},\dots, (-1)^n t_{nn} \Big) 
\]
and the eigenvalues of $SG$ are
\[
\Big(  1,-\frac{1}{2},\frac{1}{4},-\frac{1}{8},\dots, (-1)^n\frac{1}{2^n} \Big)  = \Big(  s_{00},- s_{11}, s_{22},-s_{33},\dots, (-1)^n s_{nn} \Big) .
\] 
We have a proof for this fact in Example \ref{ex:TSproof}.
\end{example}

Now, we can expect that a similar property holds for some class of triangular matrices. We first make clear the property we will study.

\begin{definition}
We say that $X =(x_{ij}) \in M_L(n+1,\C)$ has \lq\lq anti-diagonal eigenvalue property" if $XG$ is similar to $diag(x_{00},- x_{11}, x_{22},\dots, (-1)^n x_{nn})$. 
\end{definition}

\begin{definition}
We say that $X =(x_{ij}) \in M(n+1,\C)$ has \lq\lq weak anti-diagonal eigenvalue property" if $\mathrm{det}(\lambda I - XG) = \prod_{i=0}^n (\la - (-1)^i x_{ii} )$ where $I$ is the identity matrix.
\end{definition}

As stated above, $T$ and $S$ have anti-diagonal eigenvalue property. Obviously, if a matrix $X$ has anti-diagonal eigenvalue property, then it has weak anti-diagonal eigenvalue property. On the other hand, if a matrix $X =(x_{ij})$ has weak anti-diagonal eigenvalue property and satisfies $(-1)^i x_{ii} \neq (-1)^j x_{jj}$ for all $i \neq j$, then $X$ has anti-diagonal eigenvalue property. We also remark that the statement \lq\lq the set of eigenvalues of $XG$ is equal to $\{(-1)^i x_{ii} \ ;\ i=0,1, \dots,n\}$" is not equivalent to \lq\lq $XG$ has anti-diagonal eigenvalue property" nor \lq\lq $XG$ has weak anti-diagonal eigenvalue property".

To characterize triangular matrices having (weak) anti-diagonal eigenvalue property, we first study the cases $n=1$ and $2$.

\begin{lemma}\label{lem:n=1}
If $n=1$, then a triangular matrix $X$ has \lq\lq weak anti-diagonal eigenvalue property" if and only if there exists $a,b \in \C$ such that
\begin{equation}
XG = \left(
\begin{array}{cc}
 0 & a  \\
 b & a-b
\end{array}
\right).
\end{equation}
\end{lemma}
\begin{proof}
Let 
\begin{equation}
X = \left(
\begin{array}{cc}
 a & 0  \\
 c & b
\end{array}
\right)
\end{equation}
and solve the equation $\mathrm{det}(\lambda I - XG)=\lambda(\la-c)-ab=(\lambda-a)(\lambda-(-b))$ for any $\la \in \C$, then we have $c=a-b$.
\end{proof}

In particular, for any given pair of numbers $(\lambda_0,\lambda_1) \in \C^2$, there exists a unique triangular matrix $X$ satisfying $(x_{00},x_{11})=(\la_0,\la_1)$ and having \lq\lq weak anti-diagonal eigenvalue property".

\begin{lemma}\label{lem:n=2}
If $n=2$, then a triangular matrix $X$ has \lq\lq weak anti-diagonal eigenvalue property" if and only if there exists $a,b,c \in \C$ and $p,q \in \C$ satisfying $pq=2(a-b)(b-c)$ such that
\begin{equation}
XG = \left(
\begin{array}{ccc}
 0 & 0 & a   \\
 0 & b & p \\
 c & q & a-2b+c
\end{array}
\right).
\end{equation}
\end{lemma}
\begin{proof}
The same as the proof of Lemma \ref{lem:n=1}.
\end{proof}

In particular, for any given triplet of numbers $(\la_0,\la_1,\la_2) \in \C^3$, there exist uncountably many triangular matrices $X$ satisfying $(x_{00},x_{11},x_{22})=(\la_0,\la_1,\la_2)$ and having weak anti-diagonal eigenvalue property. 

On the other hand, for any triplet of numbers $(\la_0,\la_1,\la_2) \in \C^3$ satisfying $\la_0 \neq \la_1$, there exists a unique triangular matrix $X$ whose all leading principle submatrices having weak anti-diagonal eigenvalue property. Precisely, 
\begin{equation}
XG = \left(
\begin{array}{ccc}
 0 & 0 & \la_0   \\
 0 & \la_1 & \la_0-\la_1 \\
 \la_2 & 2(\la_1-\la_2) & \la_0-2\la_1+\la_2
\end{array}
\right).
\end{equation}

From the above observations, it seems too complicated to characterize all matrices having (weak) anti-diagonal eigenvalue property. Therefore, we introduce the notion of global (weak) anti-diagonal eigenvalue property and characterize matrices having this property. 

Let $\N_0=\N \cup \{0\}$ and $M(\infty, \C)=\{ X=(x_{ij})_{i,j \in \N_0} \}$ be the set of matrices with infinitely many rows and columns and $M_L(\infty, \C)=\{ X=(x_{ij})_{i,j \in \N_0}; x_{ij}=0 \ (\forall \ i  < j) \}$ the set of lower triangular matrices with infinitely many rows and columns.

\begin{definition}
We say that $X \in M_L(n+1,\C)$ or $ \in M_L(\infty, \C)$ has \lq\lq global (weak) anti-diagonal eigenvalue property" if all finite leading principal submatrices of $X$ have (weak) anti-diagonal eigenvalue property.
\end{definition}

We denote the set of matrices having global anti-diagonal eigenvalue property and global weak anti-diagonal eigenvalue property by $M^*_L(n+1,\C)$, $M^*_L(\infty,\C)$, $M^{*,weak}_L(n+1,\C)$ and $M^{*,weak}_L(\infty,\C)$ respectively.

In the rest of the paper, if a matrix $X$ is defined by its entries $(x_{ij})$ where $x_{ij}$ depends only on $i$ and $j$, then we regard $X$ as an element of $M(n+1,\C)$ as well as of $M(\infty,\C)$ depending on the context. 

\begin{remark}
Both of $T$ and $S$ in Example \ref{exm:TS} are in $M^*_L(n+1,\C)$ and also in $M^*_L(\infty,\C)$.
\end{remark}

Our goal is to give a good characterization of $M^*_L(n+1,\C)$ (or $M^*_L(\infty,\C)$) and $M^{*,weak}_L(n+1,\C)$ (or $M^{*,weak}_L(\infty,\C)$).

Let $P=(p_{ij})$ be the matrix with entries $p_{ij}=\binom{i}{j} \1_{\{i \ge j \}}$. It is easy to see that $(P^{-1})_{ij}=(-1)^{i-j} \binom{i}{j} \1_{\{i \ge j \}}$. We define the linear subspaces $V_P(n)$ of $M(n+1, \C)$ and $V_P(\infty)$ of $M(\infty, \C)$ by 
\begin{align*}
V_P(n) & =\{X \in M(n+1,\C); P^{-1}XP \ \text{is a diagonal matrix} \}, \textit{and} \\
V_P(\infty) & =\{X \in M(\infty,\C); \quad \forall n \in \N, \quad  X|_n:=(x_{ij})_{i,j=0}^{n} \in V_P(n) \},
\end{align*}
respectively.


\begin{lemma}\label{lem:Q}
For $G \in M(n+1,\C)$, $Q:=P^{-1}GP$ is an upper triangular matrix with $Q_{ij}=(-1)^i \binom{n-i}{j-i}\1_{\{j \ge i \}}$. In particular, $Q_{ii}=(-1)^i$.
\end{lemma}
\begin{proof}
We consider the correspondence between elements of $M(n+1,\C)$ and operators on polynomials of degree $n$. Define operators $T_P$ and $T_G$ on polynomials of degree $n$ as
\begin{align*}
(T_Pf)(x)=f(x+1), \quad (T_Gf)(x)=x^nf(\frac{1}{x}). 
\end{align*}
Then, the matrix representation with a basis $\{ e_i(x)=x^i;  0 \le i \le n \}$ of $T_P$ and $T_G$ are $P$ and $G$ respectively. Let $T_Q$ be the operator having the matrix representation $Q$. Then, 
\begin{align*}
(T_Qe_i)(x) & =(T_{P}T_GT_{P^{-1}}e_i) (x)  = (\frac{1}{x+1}-1)^i(x+1)^n \\
& =(-x)^i(x+1)^{n-i} = \sum_{j=0}^n (-1)^i \binom{n-i}{j-i}\1_{\{j \ge i \}} e_j(x) .
\end{align*}
Therefore, $Q_{ij}=(-1)^i \binom{n-i}{j-i}\1_{\{j \ge i \}}$.
\end{proof}

\begin{proposition}\label{prop:inj}
The following hold:
\begin{equation}
V_P(n) \subset M^{*,weak}_L(n+1, \C), \quad  V_P(\infty) \subset M^{*,weak}_L(\infty, \C).
\end{equation}
\end{proposition}
\begin{proof}
By Lemma \ref{lem:Q}, if $X \in V_P(n)$ (or $V_P(\infty)$), then $P^{-1}XGP$ is an upper triangular matrix with diagonal elements $(-1)^i x_{ii}.$ That is, $X$ has weak anti-diagonal eigenvalue property.
\end{proof}

Define the maps $\Pi: \C^{n+1} \to M_L^{*,weak}(n+1,\C)$ and $\Phi: M_L^{*,weak}(n+1,\C) \to \C^{n+1}$ by 
\begin{displaymath}
\Pi((\la_0,\la_1,\dots, \la_n)) =PDP^{-1}, \textit{where} \quad D=diag(\la_0,\la_1,\dots, \la_n), \textit{and}
\end{displaymath}
\begin{displaymath}
\Phi(X) =(x_{00},x_{11},\dots, x_{nn}), \textit{where}  \quad  X=(x_{ij})_{0 \le i,j \le n}.
\end{displaymath}
We define $\Pi: \C^{\N_0} \to M^{*,weak}_L(\infty, \C)$ and $\Phi: M^{*,weak}_L(\infty, \C) \to \C^{\N_0}$ in the same manner. 

Obviously, $\Pi$ is injective. For $\Phi$, we have a simple corollary of Proposition \ref{prop:inj}:
\begin{corollary}
$\Phi$ is surjective.
\end{corollary}
\begin{proof}
It is enough to see that $\Phi \circ \Pi =\mathrm{Id}$ and it follows by the fact that $P$ is triangular. 
\end{proof}

\begin{remark}
For $X \in V_P(n)$ or $V_P(\infty)$, $\Pi \circ \Phi (X)=X$ also holds.
\end{remark}

Next, we study the injectivity of $\Phi$. It is equivalent to study whether $M^{*,weak}_L(n+1, \C) = V_P(n)$ or not. As we saw in Lemmas \ref{lem:n=1} and \ref{lem:n=2}, if $n=1$, then $\Phi$ is injective, but if $n=2$, it is not. On the other hand, if we consider $E_2:=\{ (\la_0,\la_1,\la_2) \in \C^3; \la_0 \neq \la_1 \}$, then $\Phi : \Phi^{-1}(E_2) \to E_2$ is injective. Namely, we may be able to find a sequence of polynomials $H_k(z)=H_k(z_0,\ldots,z_{k-1}) \in \C[z_0,\ldots,z_{k-1}]$ such that if $H_k(\la_0,\la_1, \dots, \la_{k-1}) \neq 0$ for all $1 \le k \le n$, then there exists a unique matrix $X \in M_L^{*,weak}(n+1,\C)$ satisfying $\Phi(X)=(\la_0,\la_1,\dots, \la_n)$. For example, we can take $H_1(z) \equiv 1$ and $H_2(z)=H_2(z_0,z_1) =z_0-z_1$. We are going to show that it is the case.

First, we study a sufficient condition for a given sequence $(z_0,z_1,\dots,z_n)$, $X \in M^{*,weak}(n+1,\C)$ satisfying $\Phi(X)=(z_0,z_1,\dots,z_n)$ and $X|_{n-1}=P|_{n-1} D|_{n-1} P^{-1}|_{n-1}$ with $D=\textit{diag}(z_0,z_1,\dots,z_n)$ is unique. In other words, we consider a sufficient condition of $(z_0,z_1,\dots,z_n)$ for the following holds: \lq\lq If $X \in M_L^{*,weak}(n+1,\C)$ satisfies $\Phi(X)=(z_0,z_1,\dots,z_n)$ and $X=PDP^{-1}+W$ with 
\begin{equation}
W = \left(
\begin{array}{ccc}
&   &  \\
& \hsymb{0}  &\\
w_0 & w_1 \cdots  w_{n-1} & 0 \\
\end{array}
\right),
\end{equation}
then $W=0$". 

If $X \in M_L^{*,weak}(n+1,\C)$ and $\Phi(X)=(z_0,z_1,\dots,z_n)$, then 
\begin{align*}
\mathrm{det} (\lambda I - XG)= \prod_{i=0}^n (\la-(-1)^iz_i)=\mathrm{det} (\lambda I - PDP^{-1}G)
\end{align*}
holds. Therefore, for $X=PDP^{-1}+W$, we have 
\begin{align*}
\mathrm{det} (\lambda I - (PDP^{-1}+W)G)= \mathrm{det} (\lambda I - PDP^{-1}G)
\end{align*}
and it is equivalent to 
\begin{align*}
\mathrm{det} (\lambda I - P^{-1}(PDP^{-1}+W)GP) & = \mathrm{det} (\lambda I - DP^{-1}GP -P^{-1}WGP) \\ 
 & = \mathrm{det} (\lambda I - DP^{-1}GP).
\end{align*}
Let
\begin{equation*}
U:=P^{-1}WGP = \left(
\begin{array}{ccc}
&   &  \\
& \hsymb{0}  &\\
u_0 & u_1 \cdots  u_{n-1} & u_n \\
\end{array}
\right)
\end{equation*}
where $u_i=\sum_{k=0}^{n-1}w_k p_{n-k,i}$. In particular, $W=0$ if and only if $U=0$. Let $C_{j}$ be the $(n,j)$ cofactor of $\lambda I - DP^{-1}GP$. Then, by expanding determinants $\lambda I - DP^{-1}GP -U$ and $\lambda I - DP^{-1}GP$ on the $n$-th row, we have
\begin{align*}
-\sum_{j=0}^{n}u_{j} C_{j}+(\la-(-1)^{n}z_n)C_{n}=(\la-(-1)^{n}z_n)C_{n}
\end{align*}
which is equivalent to 
\begin{align}\label{eq:linear relation}
\sum_{j=0}^{n}u_{j} C_{j}=0.
\end{align}
Since $\lambda I - DP^{-1}GP=\la I -DQ$ where $Q_{ij}=(-1)^i \binom{n-i}{j-i}\1_{\{j \ge i \}}$ by Lemma \ref{lem:Q}, 
$$(\la I -DQ)_{ij}=\la\delta_{ij} - z_i(-1)^i \binom{n-i}{j-i}\1_{\{j \ge i \}},$$
that is, 
{\small
\begin{equation*}
\begin{array}{rl}
&\lambda I - DP^{-1}GP \\
=& \left(
\begin{array}{cccccccc}
\la - z_0 & -n z_0  &  \cdots & \cdots & - \binom{n}{j} z_0  & \cdots &  -z_0\\
0 &  \la+z_1 & (n-1) z_1 &  \cdots & \binom{n-1}{j-1}z_1 & \cdots & z_1\\
0 & 0 & \la-z_2 & \cdots &  \cdots & \cdots  & -z_2 \\
 && & \ddots &&& \vdots \\
 && &  &\la -(-1)^j z_{j}&\cdots& (-1)^jz_j\\
 && &  &&\ddots& \vdots \\
\hsymb{0} &&&&&&  \la -(-1)^n z_{n}\\
\end{array}
\right)
\end{array}
\end{equation*}
Therefore, the coefficient of $\la^n$ in $\sum_{j=0}^{n}u_{j} C_{j}$ is $u_n$, and so if $\sum_{j=0}^{n}u_{j} C_{j}=0$, then $u_n=0$. On the other hand, for $0 \le j \le n-1$, $C_j=\sum_{k=0}^{n-1}\la^{k} f_{k,j}(z_0,z_1,\dots,z_{n-1})$ where $f_{k,j}$ is a homogeneous polynomial of degree $n-k$ which is linear in each $z_0$, $z_1$, \dots, $z_{n-1}$. Denote the determinant of the $n \times n$ matrix having $f_{k,j}(z_0,z_1,\dots,z_{n-1})$ as the entry in row $k$ and column $l$ by $H_n(z)$, namely
\begin{equation*}
H_{n}(z)\equiv H_{n}(z_0,z_1,\dots,z_{n-1})=\det (f_{k,j}(z_0,z_1,\dots,z_{n-1}))_{k,j=0,1,\dots,n-1}.
\end{equation*} Obviously, $H_{n}(z)$ is a polynomial of $z_0$, $z_1$, \dots, $z_{n-1}$. By the construction, if $H_{n}(z) \neq 0$, then the linear system (\ref{eq:linear relation}) admits only the trivial solution $(u_0, \dots, u_n)=(0, \dots, 0)$.

\begin{lemma}\label{lem:nonzero}
For any $n \in \N_0$, $H_n(z) \not\equiv 0$.
\end{lemma}
\begin{proof}
Since $f_{k,j}(z_0,z_1,\dots,z_{n-1})$ are homogeneous polynomials of degree $n-k$ which are linear in each $z_0$, $z_1$, \dots, $z_{n-1}$, $H_n(z)$ is a homogenous polynomial of degree $\frac{n(n+1)}{2}$. We will show that the coefficient of the term $z_0^nz_1^{n-1} \cdots z_{n-1}$ is not $0$. Note that we have the following unique decomposition of $z_0^nz_1^{n-1} \cdots z_{n-1}$ into the product of $n$ monomials of degree $1$, $2$, \dots, $n$ where each of them is linear in $z_0$, $z_1$, \dots, $z_{n-1}$:
$$z_0^nz_1^{n-1} \cdots z_{n-1}=z_0(z_0z_1)(z_0z_1z_2)\cdots(z_0z_1z_2\cdots z_{n-1}).$$ Then the coefficient of the term $z_0^nz_1^{n-1} \cdots z_{n-1}$ of $H_n(z)$ is 
\begin{equation}\label{eq:coefficientofHnz}
\sum_{\sigma} \prod_{k=0}^{n-1} f^{0,1,\dots, n-k-1}_{k, \sigma(k)},
\end{equation} 
where $f^{0,1,\dots, n-k-1}_{k, j}$ is the coefficient of $z_0z_1\cdots z_{n-k-1}$ of $f_{k,j}(z)$. Now, it is easy to see that $f^0_{n-1,j}=\pm \delta_{0j}$ where $\pm$ depends on $n$. Also, for $j \ge 1$, $f^{0,1}_{n-2,j}=\pm \delta_{1j}$ and for $j \ge 2$, $f^{0,1,2}_{n-3,j}=\pm \delta_{2j}$ and so on. Thus, (\ref{eq:coefficientofHnz}) is equal to $\prod_{k=0}^{n-1} f^{0,1,\dots, n-k-1}_{k, n-k-1}=\pm1 \neq 0$.

\end{proof}

Now, we have a sufficient condition to characterize $M_L^{*,weak}(n+1,\C)$ and $M_L^{*,weak}(\infty,\C)$. Let $E_{n}:=\{ (z_0,z_1,\dots,z_n) \in \C^{n+1}; H_{k} (z_0,z_1, \dots, z_{k-1}) \neq 0, \ 1 \le \forall k \le n \}$ and $E:=\{ (z_i)_{i \in \N_0} :  H_{k} (z_0,z_1, \dots, z_{k-1}) \neq 0, \ \forall k \in \N \}$.

\begin{remark}
By Lemma \ref{lem:nonzero}, $E_n \neq \emptyset$. Moreover, $E \neq \emptyset$. In fact, since $E=\bigcap_{n=1}^{\infty}E_n$, for any probability measure $\mu$ on $\C$ with positive density function, $\mu^{\otimes \N_0}(E)=\lim_{n \to \infty} \mu^{\otimes \N_0}(E_n)=1$. 
\end{remark}

\begin{proposition}\label{prop:inj2}
$\Phi:\Phi^{-1}(E_n) \to E_n$ is injective.
\end{proposition}

\begin{proof}
We do this by the mathematical induction on $n$. The claim is true for $n=1$.

Next, assume that the claim holds for $n=k$. Fix $(\la_0,\la_1, \dots, \la_{k+1}) \in E_{k+1}$ and suppose that $X \in M_L^{*,weak}(k+2,\C)$ satisfies $\Phi(X)=(\la_0,\la_1, \dots, \la_{k+1})$. By the definition, the matrix $X|_{k}$ satisfies $\Phi(X|_{k})=(\la_0,\la_1, \dots, \la_{k})$ and $X|_k \in M_L^{*,weak}(k+1,\C)$. Since $(\la_0, \la_1,\dots, \la_k) \in E_k$, by the inductive assumption, $X|_k= PDP^{-1}$ where $D=diag(\la_0,\la_1,\dots,\la_k)$. Also, by the assumption $x_{k+1,k+1}=\la_{k+1}$ and $x_{i,k+1}=0$ for $i \le k$. Therefore, we only need to show that $x_{k+1,j},  \ 0 \le j \le k$ are determined uniquely under the condition that $X \in M_L^{*,weak}(k+2,\C)$. This follows from the definition of the set $E_{k+1}$. 

Therefore, the claim holds for all $n \in \N_0$. 
\end{proof}

\begin{proposition}
$\Phi:\Phi^{-1}(E) \to E$ is injective.
\end{proposition}
\begin{proof}
It is shown in the same way as in the proof of Proposition \ref{prop:inj2}.
\end{proof}

To characterize $M_L^{*}(n+1,\C)$ and $M_L^{*}(\infty,\C)$, we define $\tilde{E}_n$ and $\tilde{E}$ by
\begin{align*}
\tilde{E}_n & : =E_n \cap \{ (\la_0,\la_1,\dots, \la_n)  \in \C^{n+1}; (-1)^i  \la_i \neq  (-1)^j \la_j \ (i \neq j) \} \\
\tilde{E} & : =E \cap \{ (\la_n)_{n \in \N_0}  \in \C^{\N_0} ; (-1)^i \la_i \neq  (-1)^j \la_j (i \neq j) \}.
\end{align*}
As a consequence of this section, we have the following theorem.
\begin{theorem}
$\Phi : \Phi^{-1}(\tilde{E}_n) \to \tilde{E}_n$ is one-to-one and $\Phi^{-1}(\tilde{E}_n) \subset M_L^{*}(n+1,\C)$. Also, $\Phi : \Phi^{-1}(\tilde{E}) \to \tilde{E}$ is one-to-one and $\Phi^{-1}(\tilde{E}) \subset M_L^{*}(\infty,\C)$.
\end{theorem}


\subsection{Stochastic and symmetric case}

In this section, we characterize two interesting subclasses of the matrices having global (weak) anti-diagonal eigenvalue property. 

\begin{definition}
We say that $X=(x_{ij}) \in M_L(n+1,\C)$ or $M_L(\infty,\C)$ is \lq\lq stochastic" if $x_{ij} \ge 0$ for all $i,j$ and $\sum_{j=0}^{\infty} x_{ij}=(\sum_{j=0}^i x_{ij})=1$ for all $i$.
\end{definition}

\begin{definition}
We say that $X=(x_{ij}) \in M_L(n+1,\C)$ or $M_L(\infty,\C)$ is \lq\lq symmetric" if $x_{i,i-j}=x_{i,j}$ for all $i \ge j$.
\end{definition}

To make the notations simple, in this section, we only consider the infinite matrices, but similar results also hold for finite matrices.

We denote the set of stochastic lower triangular matrices by $M_{L,stoch}=M_{L,stoch}(\infty,\C)$ and the set of symmetric lower triangular matrices by $M_{L,sym}=M_{L,sym}(\infty,\C)$. We define $M^{*}_{L,stoch}=M^{*}_{L} \cap M_{L,stoch}$. $M^{*,weak}_{L,stoch}$, $M^{*}_{L,sym}$ and $M^{*,weak}_{L,sym}$ are defined by the same way. 
\begin{remark}
$T$ and $S$ in Example \ref{exm:TS} are in $M^{*}_{L,stoch} \cap M^{*}_{L,sym}$.
\end{remark}

Our goal is to characterize the sets $M^{*}_{L,stoch}$, $M^{*}_{L,sym}$ and $M^{*}_{L,stoch} \cap M^{*}_{L,sym}$. 

\subsubsection{Stochastic case}

We start from an important lemma.

\begin{lemma}\label{lem:Au}
Let $A(u)$ be the matrix with entries $A(u)_{ij}=\binom{i}{j}u^j(1-u)^{i-j}\1_{\{i \ge j \}}$. Then, $P^{-1}A(u)P= diag(1,u,u^2, \dots, u^n)$, and hence $A(u) \in V_P(n)$. In particular, the set $\{A(u_i); i =0,1,\dots,n\}$ forms a basis of $V_P(n)$ for any distinct $u_0,u_1,\dots, u_n$.
\end{lemma}
\begin{proof}
By the direct computation, 
\begin{align*}
& (P^{-1}A(u)P)_{ij} =\sum_{k=0}^n \sum_{l=0}^n (P^{-1})_{ik}A(u)_{kl}P_{lj}=  \sum_{k=0}^{i} \sum_{l=0}^{k} (P^{-1})_{ik}A(u)_{kl}P_{lj} \\
& = \1_{\{i \ge j \}} \sum_{k=0}^{i} \sum_{l=0}^{k} (-1)^{i-k} \binom{i}{k} \binom{k}{l}u^l(1-u)^{k-l}\binom{l}{j} \\
& = \1_{\{i \ge j \}} u^j \binom{i}{j} \sum_{k=0}^{i} \sum_{l=0}^{k}\frac{(i-j)!}{(i-k)!(k-l)!(l-j)!}  (-1)^{i-k} u^{l-j} (1-u)^{k-l}.
\end{align*}
Then, by multinomial theorem, $ (P^{-1}A(u)P)_{ij} = u^j  \delta_{ij}$.
\end{proof}

\begin{remark}
When $0 \le u \le 1$, $A(u)$ is a stochastic matrix.
\end{remark}

\begin{theorem}\label{thm:stoch}
Let $\mu$ be a Borel probability measure on $[0,1]$. Then the matrix $A_{\mu}:=\int_0^1A(u)\mu (du)$ is in $V_P(\infty) \cap M_{L,stoch}$ and hence in $M^{*,weak}_{L,stoch}$. In particular, if $\mathrm{support}(\mu) \not\subset  \{0,1\}$, then $A_{\mu}$ is in $M^{*}_{L,stoch}$.
\end{theorem}
\begin{proof}
By Lemma \ref{lem:Au}, we can show that 
\[
(P^{-1}A_{\mu}P)_{ij}=  \int_0^1 (P^{-1}A(u)P)_{ij} \mu (du) = \delta_{ij}\int_0^1 u ^j \mu (du).
\]
Then, since $P,A_{\mu}$ and $P^{-1}$ are triangular, $ P^{-1}|_nA_{\mu}|_nP|_n =( P^{-1}A_{\mu}P)|_n$ and so $A_{\mu} \in V_P(\infty)$. On the other hand, by the definition, it is obvious that $A_{\mu}$ is a stochastic matrix. Finally, we see that if $\mathrm{support}(\mu) \not\subset  \{0,1\}$, then $(-1)^i(A_{\mu})_{ii} \neq (-1)^j (A_{\mu})_{jj}$ for all $i \neq j$. In fact, if there exists a pair $i > j$ such that $(-1)^i (A_{\mu})_{ii} = (-1)^j (A_{\mu})_{jj}$, then
\begin{displaymath}
\int_0^1 (u^j-u^i)\mu (du) = \int_0^1 u^j(1-u^{i-j})\mu (du) =0.
\end{displaymath}
Namely, $\mathrm{support}(\mu) \subset  \{0,1\}$.
\end{proof}

\begin{example}\label{ex:TSproof}
If $\mu$ is the Lebesgue measure on $[0,1]$, then $A_{\mu}=T$. If $\mu=\delta_{1/2}$, then $A_{\mu}=S$.
\end{example}

\begin{remark}\label{rem:exchangeable}
The matrix $A_{\mu}$ appears naturally in the context of classical probability theory as an expectation of a random $0$-$1$ matrix identified with \lq\lq mixtured Bernoulli increasing process". Precisely, let $Y_1,Y_2,\ldots$ be an infinite sequence of exchangeable binary random variables, i.e., $Y_i \in \{0,1\}$ and 
\begin{equation*}
(Y_1,Y_2,\ldots) \overset{(d)}{=}(Y_{\sigma(1)},Y_{\sigma(2)},\ldots)
\end{equation*}
for any finite permutation $\sigma \in \bigcup_{n=1}^{\infty}\mathcal{S}_n$. By de Finetti's theorem, $(Y_1,Y_2,\ldots)$ is a mixture of a sequence of i.i.d. Bernoulli random variables, that is, there exists a probability measure $\mu$ on $[0,1]$ such that for all i 
\begin{equation*}
\text{Prob}(Y_1=a_1,\ldots,Y_i=a_i)=\int_0^1u^j(1-u)^{i-j} \mu(du)
\end{equation*}
when $(a_1,a_2,\ldots,a_i) \in \{0,1\}^i$ and $\sum_{k=1}^i a_k= j$. Therefore, setting $S_i=\sum_{k=1}^i Y_k$ $(S_0=0)$ we have $(A_{\mu})_{ij}=\text{Prob}(S_i=j)$, and hence $A_{\mu}=E[B]$ with $B=(1_{\{S_i=j\}})_{i,j=0}^{\infty} \in M_L(\infty,\C)$. Here $\mu$ is the limiting distribution of the sequence $\frac{S_n}{n}$ (cf.\cite{A}).
\end{remark}

By the above theorem, we can define a map $\mathcal{A}:\mathcal{P}([0,1]) \to V_P(\infty) \cap M_{L,stoch}$ as $\mathcal{A}(\mu)=A_{\mu}$ where $\mathcal{P}([0,1])$ is a set of probability measures on $[0,1]$.

Now, our interest is whether $\mathcal{A}$ is injective or not, and surjective or not. To study this problem, we introduce some notions. 
\begin{definition}
A sequence of real numbers $(a_i)_{i \in \N_0}$ is said to be completely monotone if its difference sequences satisfy the equation
\begin{equation}
(-1)^j (\Delta^j a)_i \ge 0
\end{equation}
for all $i,j \ge 0$. Here, $\Delta$ is the difference operator given by
\begin{equation}
(\Delta a)_i = a_{i+1}-a_i.
\end{equation}
\end{definition}

Now, recall a beautiful result by Hausdorff :
\begin{proposition}[Hausdorff,\cite{H}]\label{prop:Hausdorff}
For any completely monotone sequence $(a_n)_{n \in \N_0}$ satisfying $a_0=1$, there exists a unique probability measure $\mu$ on $[0,1]$ such that $\int_0^1 u^i\mu(du)=a_i$. 
\end{proposition}
Namely, we can define an injective map $\mathcal{M}: C \to \mathcal{P}([0,1])$ where $C$ denotes the set of completely monotone sequence of real numbers satisfying its first term is $1$; $C:=\{(a_i)_{i \in \N_0}; (a_i)_{i \in \N_0} \ \text{is completely monotone}, a_0=1 \}$ and $a_i=\int_0^1 u^i \mathcal{M}(a)(du), i \in \N_0$. Then, we have the following lemma.

\begin{lemma}\label{lem:id1}
$\Phi \circ \mathcal{A} \circ \mathcal{M} = \mathrm{Id}$.
\end{lemma}
\begin{proof}
By the definition, for any $a=(a_i) \in C$, $(\Phi \circ \mathcal{A} \circ \mathcal{M} (a) )_i= (\mathcal{A} \circ \mathcal{M} (a) )_{ii}= \int_{0}^1u^i\mathcal{M}(a)(du)=a_i$ for $i \in \N_0$.
\end{proof}

Next, we will show that $\Phi(V_P(\infty) \cap M_{L,stoch}) \subset C$. We prepare simple lemmas.

\begin{lemma}
For a sequence of real numbers $(a_i)_{i \in \N_0}$, 
\begin{equation}\label{eq:difference}
(\Delta^j a)_i = \sum_{k=0}^j \binom{j}{k}(-1)^{j-k}a_{k+i}.
\end{equation}
 \end{lemma}
\begin{proof}
This is proved by the mathematical induction on $j$. 
\end{proof}
\begin{lemma}\label{lem:differenceop}
Assume $X=(x_{ij}) \in V_P(\infty)$. Then, $(-1)^{i-j}x_{ij}=\1_{\{i \ge j\}}\binom{i}{j} (\Delta^{i-j} \la)_j $ where $(\la_i)_{i \in \N_0}=\Phi(X)$.
 \end{lemma}
 \begin{proof}
Since $X \in V_P(\infty)$, $X|_n= P DP^{-1}$ for some $D=diag(d_0,d_1,\dots,d_n)$. Then, by the exact computation, $x_{ii}= P_{ii}d_iP^{-1}_{ii}=d_i$. Therefore, for $i \ge j$,
\begin{align*}
x_{ij} & = (P DP^{-1})_{ij}=\sum_{k=j}^i \binom{i}{k}x_{kk} (-1)^{k-j} \binom{k}{j} \\
& =\sum_{k=0}^{i-j} \binom{i}{k+j}\binom{k+j}{j} (-1)^k \la_{k+j} =\binom{i}{j}\sum_{k=0}^{i-j} \binom{i-j}{k} (-1)^k \la_{k+j}.
\end{align*}
So, (\ref{eq:difference}) completes the proof.
\end{proof}

\begin{proposition}\label{prop:domain}
Assume $X \in V_P(\infty) \cap M_{L,stoch}$. Then, $\Phi(X) \in C$.
\end{proposition}
\begin{proof}
Since $X$ is stochastic, $x_{ij} \ge 0$ for all $i,j$. By Lemma \ref{lem:differenceop}, we have $(-1)^{i-j}(\Delta^{i-j} \la)_j \ge 0$ for all $i \ge j$ where $(\la_i)_{i \in \N_0}=\Phi(X)$ which implies the sequence $\Phi(X)$ is completely monotone. Also, since $X$ is triangular and stochastic, $X_{00}=\Phi(X)_0=1$. Namely, $\Phi(X) \in C$.
\end{proof}



Now, we have following simple relations.

\begin{proposition}\label{prop:id2}
The following hold:

(i) $\mathcal{A} \circ \mathcal{M} \circ \Phi|_{V_P(\infty) \cap M_{L,stoch}}  = \mathrm{Id}$.

(ii) $\mathcal{M} \circ \Phi|_{V_P(\infty) \cap M_{L,stoch}} \circ  \mathcal{A}  = \mathrm{Id}$.
\end{proposition}
\begin{proof}
(i) By Proposition \ref{prop:domain}, for any $X \in V_P(\infty) \cap M_{L,stoch}$, we can define $(\mathcal{A} \circ \mathcal{M} \circ \Phi)(X)$ and by the definition, $(\mathcal{A} \circ \mathcal{M} \circ \Phi)(X) \in V_P(\infty)$. Therefore, to show $X= (\mathcal{A} \circ \mathcal{M} \circ \Phi)(X)$, we only need to show that $\Phi(X)=  (\Phi \circ \mathcal{A} \circ \mathcal{M} \circ \Phi)(X)$. Since $\Phi( (\mathcal{A} \circ \mathcal{M} \circ \Phi)(X))_i= (\mathcal{A} \circ \mathcal{M} \circ \Phi)(X)_{ii}= \int_0^1 u^i (\mathcal{M} (\Phi(X) ))(du) =  \Phi(X)_i$ for $i \in \N_0$. Therefore, we have for any $X \in V_P(\infty) \cap M_{L,stoch}$, $(\mathcal{A} \circ \mathcal{M} \circ \Phi)(X)=X$.

(ii) For $\mu \in \mathcal{P}([0,1])$, to show $\mu= (\mathcal{M} \circ \Phi|_{V_P(\infty) \cap M_{L,stoch}} \circ  \mathcal{A}) (\mu)$, it is enough to prove that $\int_0^1 u^i \mu (du) = \int_0^1 u^i (\mathcal{M} \circ \Phi|_{V_P(\infty) \cap M_{L,stoch}} \circ  \mathcal{A}) (\mu) (du)$ since $\mathcal{M}$ is a map. By the definition, $ \int_0^1 u^i (\mathcal{M} \circ \Phi|_{V_P(\infty) \cap M_{L,stoch}} \circ  \mathcal{A}) (\mu) (du) = (\Phi|_{V_P(\infty) \cap M_{L,stoch}} \circ  \mathcal{A}) (\mu)_i =  \mathcal{A} (\mu)_{ii} = \int_0^1 u^i \mu(du)$, and the proof is completed.

\end{proof}

\begin{remark}
Combining Lemma \ref{lem:id1} and Proposition \ref{prop:id2}, we have $\mathcal{A}$, $\mathcal{M}$ and $\Phi : V_P(\infty) \cap M_{L,stoch} \to C$ are all bijections.
\end{remark}

Finally, we have a characterization of matrices as follows.

\begin{theorem}
For $X \in M_L(\infty,\C)$, the following conditions are equivalent:

(i) $X \in V_P(\infty) \cap M_{L,stoch}$.

(ii) $X= A_{\mu}$ for some $\mu \in \mathcal{P}([0,1])$.

(iii) $X= \Pi ( \la)$ for some $\la=(\la_i)_{i \in \N_0} \in C$.

Moreover, if the above conditions are satisfied, then $X \in M^{*,weak}_{L,stoch}$ and the probability measure $\mu$ given in (ii) is unique and $\mathcal{M}(\la)=\mu$ where $\la$ is the sequence given in (iii).
\end{theorem}

\begin{theorem}
Let $\la=(\la_n)_{n \in \N_0} \in E$. Then, for $X=(x_{ij}) \in M_L(\infty,\C)$, the following conditions are equivalent:

(i) $X \in M^{*,weak}_{L,stoch}$ and $\Phi(X)=\la$.

(ii) $\la \in C$ and $X= A_{\mu}$ where $\mu=\mathcal{M}(\la)$.

(iii) $\la \in C$ and $X= \Pi(\la)$.

Moreover, if $\la \in \tilde{E}$, then if the above conditions are satisfied, $X \in M^*_{L,stoch}$.
\end{theorem}






\begin{remark}
We can show the results in this subsection by considering the correspondence between $M(\infty,\C)$ and operators on polynomials. Namely, we define operators on polynomials $T_{\mu}$ for $\mu \in \mathcal{P}([0,1])$ and $T_P$ as
\begin{align*}
(T_{\mu}f)(x)=\int_0^1 f(xu+1-u)\mu(du), \quad (T_Pf)(x)=f(x+1).
\end{align*}
Then, the matrix representation of $T_{\mu}$ is $A_{\mu}$ and $T_P$ is $P$. Then, it is easy to see that
\begin{align*}
(T_{P}T_{\mu}T_{P^{-1}}f) (x) = \int_0^1 f(xu)\mu(du).
\end{align*}
From this expression, we see that $e_i(x)=x^i$ is an eigenfunction corresponding to the eigenvalue $\int_0^1 u^i\mu(du)$.  
\end{remark}

\subsubsection{Stochastic and symmetric case}

Next, we consider the stochastic and symmetric case. Here, we start from lemmas again.

\begin{lemma}
$A_{\mu} \in M_{L,sym}$ if and only if $\mu$ is invariant under the reflection with respect to $\frac{1}{2}$.
\end{lemma}
\begin{proof}
If $\mu$ is invariant under the reflection with respect to $\frac{1}{2}$, then for any $i \ge j$, $(A_{\mu})_{ij}= \binom{i}{j} \int_0^1 u^j (1-u)^{i-j} \mu(du)=  \binom{i}{j} \int_0^1 (1-u)^j u^{i-j} \mu(du) =  \binom{i}{i-j}  \int_0^1 u^{i-j} (1-u)^{i-(i-j)}  \mu(du) = (A_{\mu})_{i,i-j}$, so $A_{\mu} \in M_{L,sym}$. To show the opposite, assume $A_{\mu} \in M_{L,sym}$. Denote the reflection of $\mu$ with respect to $\frac{1}{2}$ by $\bar{\mu}$. Then, $\int_0^1 u^i \bar{\mu} (du) = \int_0^1 (1-u)^i \mu(du) = (A_{\mu})_{i0} = (A_{\mu})_{ii}$ where the last equality follows by the assumption $A_{\mu} \in M_{L,sym}$. Therefore, $\mathcal{M}^{-1}(\mu)=\mathcal{M}^{-1} (\bar{\mu})$ which implies $\mu=\bar{\mu}$.
\end{proof}

Define a subset of $\C^{\N_0}$ as follows :
\begin{displaymath}
D:=\{ \la=(\la_i)_{i \in \N_0} ; \la_{2i+1}=\frac{1}{2} \sum_{k=0}^{2i} \binom{2i+1}{k} (-1)^k\la_k, \forall i \in \N_0 \}.
\end{displaymath}

\begin{remark}
There is a natural bijection between $D$ and $\{ (\la_i)_{i \in 2\N_0} \}$.
\end{remark}

\begin{lemma}
For $\la \in C$, $\mathcal{M}(\la)$ is invariant under the reflection with respect to $\frac{1}{2}$ if and only if $\la \in D$.
\end{lemma}
\begin{proof}
If $\mathcal{M}(\la)$ is invariant under the reflection with respect to $\frac{1}{2}$, then 
\begin{align*}
\la_{2i+1} &= \int_0^1 u^{2i+1} \mathcal{M}(\la) (du) = \int_0^1 (1-u)^{2i+1} \mathcal{M}(\la) (du) \\
& =\sum_{k=0}^{2i+1}  \binom{2i+1}{k} (-1)^k \int_0^1 u^{k} \mathcal{M}(\la) (du) = \sum_{k=0}^{2i+1}  \binom{2i+1}{k} (-1)^k \la_k.
\end{align*}
Therefore, $\la \in D$. On the other hand, if $\la \in D$, then $\int_0^1 u^{2i+1} \mathcal{M}(\la) (du) = \int_0^1 (1-u)^{2i+1} \mathcal{M}(\la) (du)$ for $i \in \N_0$ as above. Moreover, it is obvious that $\int_0^1 u^{0} \mathcal{M}(\la) (du) = \int_0^1 (1-u)^{0} \mathcal{M}(\la) (du)$. Now, if we know that $\int_0^1 u^{k} \mathcal{M}(\la) (du) = \int_0^1 (1-u)^{k} \mathcal{M}(\la) (du)$ for all $0 \le k \le 2i-1$, then
\begin{align*}
& \int_0^1 (1-u)^{2i} \mathcal{M}(\la) (du)  = \sum_{k=0}^{2i-1} \binom{2i}{k} (-1)^k \int_0^1 u^k \mathcal{M}(\la) (du) + \int_0^1 u^{2i} \mathcal{M}(\la) (du) \\
 & = \sum_{k=0}^{2i-1} \binom{2i}{k} (-1)^k \int_0^1 (1-u)^k \mathcal{M}(\la) (du) + \int_0^1 u^{2i} \mathcal{M}(\la) (du) \\
  & =  \int_0^1 (1-(1-u))^{2i} \mathcal{M}(\la) (du) -  \int_0^1 (1-u)^{2i} \mathcal{M}(\la) (du)  + \int_0^1 u^{2i} \mathcal{M}(\la) (du). 
\end{align*}
Namely, $\int_0^1 u^{k} \mathcal{M}(\la) (du) = \int_0^1 (1-u)^{k} \mathcal{M}(\la) (du)$ holds for $k=2i$. Then, by the mathematical induction, $\int_0^1 u^{k} \mathcal{M}(\la) (du) = \int_0^1 (1-u)^{k} \mathcal{M}(\la) (du)$ holds for all $k \in \N_0$ which implies $\mathcal{M}(\la)$ is invariant under the reflection with respect to $\frac{1}{2}$.
\end{proof}

Then, next follows straightforwardly.

\begin{proposition}
For $X \in M_L(\infty,\C)$, the following conditions are equivalent:

(i) $X \in V_P(\infty) \cap M_{L,stoch} \cap M_{L,sym}$.

(ii) $X= A_{\mu}$ for some $\mu \in \mathcal{P}([0,1])$ which is invariant under the reflection with respect to $\frac{1}{2}$.

(iii) $X = \Pi(\la)$ for some $\la=(\la_i)_{i \in \N_0} \in C \cap D$. 

Moreover, if the above conditions are satisfied, then $X \in M^{*,weak}_{L,stoch} \cap M^{*,weak}_{L,sym}$ and the probability measure $\mu$ given in (ii) is unique and $\mathcal{M}(\la)=\mu$ where $\la$ is the sequence given in (iii).
\end{proposition}

\begin{theorem}
Let $\la=(\la_n)_{n \in \N_0} \in E$. Then, for $X \in M_L(\infty,\C)$, the following conditions are equivalent:

(i) $X \in M^{*,weak}_{L,stoch} \cap M^{*,weak}_{L,sym}$ and $\Phi(X)=\la$.

(ii) $\la \in C \cap D$ and $X= A_{\mu}$ where $\mu=\mathcal{M}(\la)$.

(iii) $\la \in C \cap D$ and $X= \Pi(\la)$.

Moreover, if $\la \in \tilde{E}$, then if the above conditions are satisfied, $X \in M^*_{L,stoch} \cap M^*_{L,sum}$.
\end{theorem}

\subsubsection{Symmetric case}

Finally, we consider the symmetric case. In this subsection, we first study the finite dimensional case.

Let $D(n):=\{ (\la_0,\la_1,\dots,\la_n) ;  \la_{2i+1}=\frac{1}{2} \sum_{k=0}^{2i} \binom{2i+1}{k} (-1)^k\la_k, \forall i \le \frac{n-1}{2} \}$.

\begin{lemma}\label{lem:sym}
Let $B(u)=\frac{1}{2}(A(u)+A(1-u))$. Then, $B(u) \in V_P(n+1) \cap M_{L,sym}(n+1,\C)$. In particular, if $n$ is odd, then the set $\{ B(u_i), i=0,1,\dots, \frac{n-1}{2} \} $ forms a basis of $V_P(n+1,\C) \cap M_{L,sym}(n+1,\C)$ whenever $|\{u_i,1-u_i; i=0,1,\dots, \frac{n-1}{2} \}|=n+1$. Also, if $n$ is even, then the set $\{ B(u_i), i=0,1,\dots, \frac{n-2}{2} \} \cup \{B(\frac{1}{2})\} $ forms a basis of $V_P(n+1,\C) \cap M_{L,sym}(n+1,\C)$ whenever $|\{u_i,1-u_i; i=0,1,\dots, \frac{n-2}{2} \}|=n$.
\end{lemma}
\begin{proof}
Since $A(u), A(1-u) \in V_P(n+1)$ and $V_P(n+1)$ is a linear space, $B(u) \in V_P(n+1)$. Also, if $i \ge j$, then $B(u)_{ij}=\frac{1}{2}A(u)_{ij} + \frac{1}{2}A(1-u)_{ij}= \binom{i}{j}  \frac{1}{2}(u^j(1-u)^{i-j}+(1-u)^j u^{i-j})=  \binom{i}{i-j}\frac{1}{2}( u^{i-j} (1-u)^j + (1-u)^{i-j} u^j) =\frac{1}{2} A(u)_{i,i-j} + \frac{1}{2}A(1-u)_{i,i-j} =B(u)_{i,i-j}$. Therefore, $B(u) \in V_P(n+1) \cap M_{L,sym}(n+1,\C)$. 

Next, we will show that the dimension of the linear space $V_P(n+1,\C) \cap M_{L,sym}(n+1,\C)$ is $[\frac{n+2}{2}]$. Since $\{A(v_i),i=0,1,\dots,n\}$ is linearly independent for any distinct $v_0,v_1,\dots,v_n$, if $n$ is odd, then $\{ B(u_i), i=0,1,\dots, \frac{n-1}{2} \} $ is linearly independent whenever $|\{u_i,1-u_i; i=0,1,\dots, \frac{n-1}{2} \}|=n+1$. Namely, $\mathrm{dim}(V_P(n+1,\C) \cap M_{L,sym}(n+1,\C)) \ge  [\frac{n+2}{2}]$. For the same reason, $\mathrm{dim}(V_P(n+1,\C) \cap M_{L,sym}(n+1,\C)) \ge  [\frac{n+2}{2}]$ holds also for even $n$. Now, we only need to show that $\mathrm{dim}(V_P(n+1,\C) \cap M_{L,sym}(n+1,\C)) \le  [\frac{n+2}{2}]$. Here, we see that if $X \in V_P(n+1) \cap M_{L,sym}(n+1,\C)$, then $\Phi(X) \in D(n)$ since
\begin{align*}
\Phi(X)_{2i+1}=x_{2i+1,2i+1}= x_{2i+1,0}= \sum_{k=0}^{2i+1} P_{2i+1,k}x_{kk}P^{-1}_{k,0}= \sum_{k=0}^{2i+1} \binom{2i+1}{k} \Phi(X)_k (-1)^k.
\end{align*} 
As $\Phi :  V_P(n+1) \cap M_{L,sym}(n+1,\C) \to D(n)$ is an injective linear map and the dimension of $D(n)$ is $[\frac{n+2}{2}]$, we complete the proof.
\end{proof}
From the result, we have a simple corollary.
\begin{corollary}
$\Phi :  V_P(n+1) \cap M_{L,sym}(n+1,\C) \to D(n)$ is an isomorphism and its inverse is $\Pi$.
\end{corollary}

\begin{theorem}
For $X \in M_L(\infty,\C)$, the following conditions are equivalent:

(i) $X \in V_P(\infty) \cap M_{L,sym}$.

(ii) $X = \Pi(\la)$ for some $\la=(\la_i)_{i \in \N_0} \in D$. 

Moreover, if the above conditions are satisfied, then $X \in M^{*,weak}_{L,sym}$.
\end{theorem}
\begin{proof}
First assume (i) holds. Then, we only need to show that $\Phi(X) \in D$ and it follows from the proof of Lemma \ref{lem:sym}.
Next, assume (ii) holds. Then, by the definition of $\Pi$, $X|n=\Pi ((\la_0,\la_1,\dots,\la_n)) \in V_P(n+1) \cap M_{L,sym}(n+1,\C)$ for any $n$. Namely, $X \in V_P(\infty) \cap M_{L,sym}$.
\end{proof}

\begin{theorem}
Let $\la=(\la_n)_{n \in \N_0} \in E$. Then, for $X \in M_L(\infty,\C)$, the following conditions are equivalent:

(i) $X \in M^{*,weak}_{L,sym}$ and $\Phi(X)=\la$.

(ii) $\la \in D$ and $X= \Pi(\la)$.

Moreover, if $\la \in \tilde{E}$ and the above conditions are satisfied, then $X \in M^*_{L,sym}$.
\end{theorem}

\section{Interacting particle process}

In this section, we consider a special matrix obtained from an invariant probability measure of interacting particle systems.

The interacting particle system is a Markov process describing a dynamics of \lq\lq particles" moving on discrete sites. Its state space is $\N_0^L$ where $L$ is the number of sites. An element $\eta=(\eta_x)_{1 \le x \le L} \in \N_0^L$ represents the configuration of particles, namely $\eta_x$ represents the number of particles at a site $x$.

We are interested in a conservative interacting particle system. Precisely, we consider a Markov process where the sum of the numbers of particles $\sum_{x=1}^L\eta_x$ is conserved. In other words, the configuration space $\Sigma_{L,n}:=\{ \eta \in \N_0^L; \sum_{x=1}^L \eta_x =n\}$ is invariant under our dynamics. In particular, we consider a process having product invariant measure. Let $\nu$ be a probability measure on $\N_0$ and $\nu^{\otimes L}$ be the product of $\nu$ on $\N_0^L$. We assume that $\nu(i) > 0$ for all $i \in \N_0$ throughout the paper. Then, if the process is conservative and $\nu^{\otimes L}$ is an invariant measure, then $\nu_{L,n}=\nu^{\otimes L}|_{\Sigma_{L,n}}$ is also an invariant measure. 

We also assume that our dynamics on $\Sigma_{L,n}$ is ergodic. Then the distribution of the process at time $t$ converges to $\nu_{L,n}$ as $t \to \infty$ independently of the initial distribution. The speed of this convergence is one of the main interest in the study of Markov processes and also is important when we study the scaling limit of our Markov process. This speed is estimated by the spectral gap, which is the smallest non-zero eigenvalue of a certain operator defined by the generator of the process. The spectral gap generally depends on $L$ and $n$. 

A sharp estimate of the spectral gap in terms of $L$ and $n$ is essential in the study of the hydrodynamic limit, which is one of the important scaling limits motivated by the rigorous study of statistical mechanics. So, there are many works on it (Cf. \cite{KL,LSV,M,NS}). Recently, Caputo \cite{C} introduced a new and elementary method, and Sasada \cite{S} showed that the method is applicable to a fairly general class of models. The key step of this method is to give a sharp estimate of the spectral gap for $L=3$ case, and for this purpose, we need to estimate the eigenvalues of the following matrices $R_n= (r^{(n)}_{ij})_{i,j=0}^n \in M(n+1,\R)$ where
\begin{equation}
r^{(n)}_{ij}:=\nu_{3,n}(\eta_2=j | \eta_1=i)=\frac{\nu_{3,n}(\eta_1=i, \eta_2=j)}{\nu_{3,n}(\eta_1=i)} = \nu_{2,n-i}(\eta_2=j).
\end{equation}
Then, it is easy to see that
\begin{align*}
r^{(n)}_{ij}=0 \quad \text{if} \  i+j >n, \quad 
\sum_{j=0}^n r^{(n)}_{ij}=1 \quad \forall \ i, \quad
r^{(n)}_{ij}=r^{(n)}_{i, n-i-j} \quad \forall \ i+j \le n . 
\end{align*}
Actually, $R_n$ is the upper anti-triangular matrix. On the eigenvalues of $R_n$, we have a simple lemma.
\begin{lemma}
For any $n$, all eigenvalues of $R_n$ are real. Moreover, since $R_n$ is an irreducible stochastic matrix, $1$ is an eigenvalue of $R_n$ and any other eigenvalue $\la$ is smaller than or equal to $1$ in absolute value, $|\la|\le 1$. 
\end{lemma}
\begin{proof}
The latter part of the claim follows from the Perron-Frobenius theorem. The former part follows from the fact $R_n$ is a self-adjoint operator on $L^2(\pi_1 \circ \nu_{3,n})$ where $\pi_1 : \Sigma_{3,n} \to \{0,1,\dots,n\}$ is $\pi_1(\eta)=\eta_1$ and $R_n : L^2(\pi_1 \circ \nu_{3,n}) \to L^2(\pi_1 \circ \nu_{3,n})$ is defined as $R_nf (j)=\sum_{k=0}^n r^{(n)}_{jk}f(k)$ for $f : \{0,1,\dots,n\} \to \R$. We have
\begin{equation*}
E_{\nu_{3,n}}[(R_nf) (\eta_1) g(\eta_1)]= E_{\nu_{3,n}}[f (\eta_1) (R_ng)(\eta_1)]
\end{equation*}
for any $f$ and $g$ since $R_nf (\eta_1)=E_{\nu_{3,n}}[f(\eta_2) | \eta_1]$, therefore $R_n$ is self-adjoint.
\end{proof}

From the observation of \cite{S}, we need to find a sufficient condition on $\nu$ for the following hold:
\begin{align}
& \inf_n \min \{ \text{eigenvalues of} \ R_n\} > -1, \label{eq:conditioninf}  \\
&  \sup_n \max (\{ \text{eigenvalues of} \ R_n \} \setminus \{1\}) < \frac{1}{2}. \label{eq:conditionsup}
\end{align}

\begin{lemma}
For any $\nu$ satisfying $\nu(i) >0$ for all $i \in \N_0$, $\inf_n \min \{ \text{eigenvalues of} \ R_n\} \ge -\frac{1}{2}$. In particular, (\ref{eq:conditioninf}) holds.
\end{lemma}
\begin{proof}
For any $f : \N_0 \to \R$, since $\nu_{3,n}$ is invariant under the exchange of sites, we have 
\begin{align*}
& E_{\nu_{3,n}}[f (\eta_1) f(\eta_1)] + 2 E_{\nu_{3,n}}[f (\eta_1) (R_nf) (\eta_1)] = E_{\nu_{3,n}}[f (\eta_1) f(\eta_1)] + 2 E_{\nu_{3,n}}[f (\eta_1) f (\eta_2)] \\
& = \frac{1}{3} \sum_{i=1}^3  \sum_{j=1}^3  E_{\nu_{3,n}}[f (\eta_i) f(\eta_j)] = \frac{1}{3}  E_{\nu_{3,n}}[(\sum_{i=1}^3f(\eta_i))^2] \ge 0.
\end{align*}
Namely, $\min \{ \text{eigenvalues of} \ R_n\} \ge -\frac{1}{2}$.
\end{proof}

It is not simple to find a sufficient condition for (\ref{eq:conditionsup}). So, we use the results in the preceding section.

Define $\tilde{R}_{\nu}= (\tilde{r}_{ij}) \in M(\infty,\C)$ as 
\begin{equation}\label{eq:r}
\tilde{r}_{ij}=\nu_{2,i}(\eta_1=j)=r^{(n)}_{n-i,j},
\end{equation}
then
\begin{align*}
\tilde{r}_{ij}=0 \quad \text{if} \  i < j, \quad 
\sum_{j=0}^{\infty} \tilde{r}_{ij}=1 \quad \forall \ i, \quad
\tilde{r}_{ij}=\tilde{r}_{i,i-j} \quad \forall \ i \ge j. 
\end{align*}
Namely, $\tilde{R}_{\nu} \in M_{L,stoch} \cap M_{L,sym}$. If $\tilde{R}_{\nu}$ has global weak anti-diagonal eigenvalue property, then the eigenvalues of $R_n=G \tilde{R}|_n$ are $((-1)^i \tilde{r}_{ii})$, and hence 
\begin{equation}\label{eq:conditiontilder}
\sup_n  (\max \{ \text{eigenvalues of} \ R_n\} \setminus \{1\}) = \sup_i \{ (-1)^i \tilde{r}_{ii} ; i \ge 2 \}.
\end{equation}

Therefore, our goal is to characterize the probability measures $\nu$ such that the associated $\tilde{R}_{\nu}$ has global weak anti-diagonal eigenvalue property and satisfies (\ref{eq:conditiontilder}). By Proposition \ref{prop:inj}, $\tilde{R}_{\nu} \in V_P(\infty)$ is a sufficient condition for the former property.



Now, we introduce some notations. For any given function $g:\N  \to (0,\infty)$, let 
\begin{equation}
Z_i=\sum_{j=0}^i \binom{g(i)}{g(j)}
\end{equation}
and a matrix $A_g \in M(\infty,\C)$ be
\begin{equation}
(A_{g})_{ij}=\frac{1}{Z_i} \binom{g(i)}{g(j)} \mathbf{1}_{\{ i \ge j\}}
\end{equation}
where $\binom{g(i)}{g(j)}=\frac{g(i)!}{g(j)!g(i-j)!}$, $g(i)!=g(i)g(i-1) \cdots g(1)$ for $i \in \N$ and $g(0)!=1$.
Note that these notations are different from the usual binomial coefficients or factorial.

Since we assume that $\nu(i) >0$ for all $i \in \N_0$, we can define $g_{\nu}:\N  \to (0,\infty)$ as $g_{\nu}(i)=\frac{\nu(i-1)}{\nu(i)}$, and for any $i \ge j$
\begin{equation}
\tilde{r}_{ij}=\frac{\nu(j)\nu(i-j)}{\sum_{k=0}^i\nu(k)\nu(i-k)}= \Big( \sum_{k=0}^i\frac{\nu(k)\nu(i-k)}{\nu(i)} \Big)^{-1}\frac{\nu(j)\nu(i-j)}{\nu(i)} =\frac{1}{Z_i} \binom{g_{\nu}(i)}{g_{\nu}(j)} \end{equation}
by the definition (\ref{eq:r}). Namely, $\tilde{R}_{\nu}=A_{g_{\nu}}$.
Then, the condition (\ref{eq:conditiontilder}) is equivalent to $\inf \{ Z_{2i}; i \ge 1 \} >2$. Since $Z_{2i} \ge 2+ \binom{g_{\nu}(2i)}{g_{\nu}(1)} = 2+\frac{g_{\nu}(2i)}{g_{\nu}(1)}$, it is equivalent to $ \inf \{ g_{\nu}(2i) ; i \ge 1 \} >0$.


So, we first study a sufficient and necessary condition for $A_g \in V_P(\infty)$ in terms of $g$ where $g$ is not necessarily given as $g_{\nu}$ by some probability measure $\nu$.


\begin{remark}
In the case $g(i)=\a$ and $g(i)=\a i$ for some constant $\a>0$, the matrix $A_g$ is equal to $T$ and $S$ in Example \ref{exm:TS}, respectively. In terms of $\nu$, in the case $\nu$ is a geometric distribution and a poisson distribution, the matrix $\tilde{R}_{\nu}=A_{g_{\nu}}$ is equal to $T$ and $S$ in Example \ref{exm:TS}, respectively.
\end{remark}

\begin{remark}
The function $g$ has an important meaning in the zero-range process, which is one of the most important interacting particle systems. The positive number $g(i)$ represents the \lq\lq jump rate" of a particle under the condition that there are $i$ particles at the same site of the particle. We refer to \cite{KL} for the precise description of zero-range processes. For the known results on the spectral gap estimates for zero-range processes, we refer to \cite{LSV,M}.
\end{remark}

\begin{lemma}
Let $c >0$ and $\tilde{g}(i)=c g(i)$. Then $A_{g}=A_{\tilde{g}}$.
\end{lemma}
\begin{proof}
Straightforward.
\end{proof}

So, from now on we only consider a function $g$ satisfying $g(1)=1$. 

\begin{theorem}\label{thm:unique}
For each $s \in (0,\infty)$, there can be at most one function $g$ satisfying $g(1)=1$, $g(2)=s$ and $A_g \in V_P(\infty)$.
\end{theorem}

\begin{proof}
By the direct computation, if $A_g \in V_P(\infty)$, we have
\begin{align*}
 (P^{-1}A_g P)_{i,i-1} & =-i\frac{1}{Z_{i-1}}\binom{g(i-1)}{g(i-1)} + \frac{1}{Z_{i}}\binom{g(i)}{g(i-1)} + i\frac{1}{Z_i}\binom{g(i)}{g(i)} \nonumber  \\ 
& 
 = -\frac{i}{Z_{i-1}} + \frac{1}{Z_{i}} g(i) + \frac{i}{Z_i} = 0 
\end{align*}
which should be rewritten as 
\begin{equation}
\big(g(i)+i\big)Z_{i-1}=iZ_i \label{eq:Z}
\end{equation}
for any $i \ge 2$.

Let $\tilde{Z}_i:=\sum_{j=1}^i\frac{g(i)!}{g(j)!g(i+1-j)!}$. Note that $\tilde{Z_i}$ depends only on $g(1), g(2), \dots, g(i)$. Then, we have 
\begin{equation*}
Z_i=\sum_{j=0}^i \binom{g(i)}{g(j)}=2+ \sum_{j=1}^{i-1} \binom{g(i)}{g(j)}=2+g(i) \sum_{j=1}^{i-1}\frac{g(i-1)!}{g(j)!g(i-j)!}=2+g(i)\tilde{Z}_{i-1}
\end{equation*}
for $i \ge 2$. Therefore, (\ref{eq:Z}) is rewritten as
$\big(g(i)+i\big)Z_{i-1}=i(2+g(i)\tilde{Z}_{i-1})$
and therefore 
\begin{equation}
g(i)(i\tilde{Z}_{i-1} - Z_{i-1})=i(Z_{i-1}-2). \label{eq:relation}
\end{equation}
Given $i \ge 3$ and $g(1), g(2), \dots g(i-1) >0 $, it is obvious that $Z_{i-1} >2$. Therefore $g(i)>0$ satisfying (\ref{eq:Z}) exists if and only if $i\tilde{Z}_{i-1} - Z_{i-1} >0$ and it is determined uniquely by $g(1), g(2), \dots g(i-1)$.
\end{proof}

\begin{theorem}\label{thm:beta}
Given $0<s<2$, let $t=\frac{s}{2-s}$ and $g(i)=\frac{it}{i+t-1}=\frac{is}{2(i-1)-(i-2)}$. Then, $g(1)=1$, $g(2)=\frac{2t}{t+1}=s$ and $A_g \in V_P(\infty)$.
\end{theorem}
\begin{proof}
Let $\mu_{t}$ be the beta distribution with probability density $B(t,t)^{-1}x^{t-1}(1-x)^{t-1}dx$ on $[0,1]$. We will show that $A_{\mu_t}=A_g$. Then $A_g \in V_P(\infty)$ follows from Theorem \ref{thm:stoch}.

By the direct computation, for $i \ge j$
\begin{align*}
& (A_{\mu_t})_{ij}=\frac{1}{B(t,t)}\binom{i}{j} \int_0^1 x^{j+t-1} (1-x)^{i-j+t-1} dx \\
&= \frac{\Gamma(2t)}{\Gamma(t)\Gamma(t)}\frac{\Gamma(i+1)}{\Gamma(j+1)\Gamma(i-j+1)}B(j+t,i-j+t) \\
& =\frac{\Gamma(2t)\Gamma(i+1)\Gamma(j+t)\Gamma(i-j+t)}{\Gamma(t)\Gamma(t)\Gamma(j+1)\Gamma(i-j+1)\Gamma(i+2t)}.
\end{align*}
On the other hand, $g(n)!=\frac{ n! t^n}{(n+t-1)(n+t-2) \cdots t}=t^n\frac{\Gamma(n+1)\Gamma(t)}{\Gamma(n+t)}$.
Therefore, for $i \ge j$
\begin{equation}
(A_g)_{ij}=\frac{1}{Z_i} \binom{g(i)}{g(j)} = \frac{1}{Z_i} \frac{\Gamma(i+1)\Gamma(t)}{\Gamma(i+t)} \frac{\Gamma(j+t)}{\Gamma(j+1)\Gamma(t)}\frac{\Gamma(i-j+t)}{\Gamma(i-j+1)\Gamma(t)}.
\end{equation}
To prove $A_{\mu_t}=A_g$ we only need $Z_i= \frac{\Gamma(2t)\Gamma(i+t)}{\Gamma(t)\Gamma(i+2t)}=\frac{(2t)_i}{(t)_i}$ where $(t)_i= \prod_{k=0}^{i-1}(k+t)$. By the definition of $A_{\mu_t}$ and $A_g$, $\sum_{j=0}^i (A_{\mu_t})_{ij} = 1= \sum_{j=0}^i (A_g)_{ij}$ for any $i \in\N_0$, so we complete the proof.
\end{proof}

\begin{theorem}
Let $g(i)=i$. Then $g(1)=1, g(2)=2$ and $A_g \in V_P(\infty)$.
\end{theorem}
\begin{proof}
It follows from $A_g=A_{\delta_{\frac{1}{2}}}$ straightforwardly.
\end{proof}

\begin{theorem}
Let $s >2$. Then there does not exist any function $g$ such that $g(1)=1$, $g(2)=s$ and $A_g \in V_P(\infty)$.
\end{theorem}
\begin{proof}
Let us consider a sequence of rational functions $\{f_i\}_{i \in \N}$ on $\C$ defined as $f_i(z)=\frac{iz}{2(i-1)-(i-2)z}$. Let $\{V_i\}$ and $\{\tilde{V}_i\}$ be the sequences of rational functions as
\begin{equation}
V_i(z)=\sum_{j=0}^i \binom{f_i(z)}{f_j(z)}=\sum_{j=0}^i \frac{f_i(z)!}{f_j(z)!f_{i-j}(z)!}
\end{equation}
\begin{equation}
\tilde{V}_i(z)=\sum_{j=1}^i\frac{f_i(z)!}{f_j(z)!f_{i+1-j}(z)!}
\end{equation}
where $f_i(z)!=f_i(z)f_{i-1}(z) \cdots f_1(z)$. Then, by Theorem \ref{thm:beta}, $f_i(z)(i\tilde{V}_{i-1}(z) - V_{i-1}(z))=i(V_{i-1}(z)-2)$ holds for $z \in (0,2)$. Then, by the identity theorem, $f_i(z)(i\tilde{V}_{i-1}(z) - V_{i-1}(z))=i(V_{i-1}(z)-2)$ for all $z \in \C$. 

Now, for given $s >2$, assume that there exists a function $g:\N \to (0,\infty)$ such that $g(1)=1$, $g(2)=s$ and $A_g \in V_P(\infty)$. Then, since $g(1)=f_1(s)$, $g(2)=f_2(s)$ and $g(i)$ must satisfy (\ref{eq:relation}) by Theorem \ref{thm:unique}, for all $i\in \N$, $g(i)=f_i(s)$. However, we have $f_i(s)  <0$ for large enough $i$, which implies a contradiction. 
\end{proof}

From the preceding theorems, we have a sufficient and necessary condition for $A_g \in V_P(\infty)$ in terms of $g$.
\begin{theorem}\label{thm:gcondition}
Let $g: \N \to (0,\infty)$. Then, $A_g \in V_P(\infty)$ if and only if $g=\a G_t$ for some $t>0$ and $\a >0$ where $G_t(i)=\frac{it}{i+t-1}$ or $g=\a Id$ for some $\a>0$. 
\end{theorem}

\begin{remark}
By Theorem \ref{thm:gcondition} and the proof of Theorem \ref{thm:beta}, we have 
\begin{align*}
\{ A_g ; & \ g :\N \to (0,\infty)\} \cap V_P(\infty) \\
& = \{ A_{\mu_t}; \ \mu_t \ \text{is the beta distribution with parameters} \ (t,t), t \in (0,\infty] \}.
\end{align*}
From the viewpoint of infinite exchangeable sequences in Remark \ref{rem:exchangeable}, the above class is exactly the same as those obtained form the classical P\'olya urn model. Precisely, consider the P\'olya urn initially with same number of black and white balls. Then, we get an infinite sequence of random variables $Y_1,Y_2,\ldots$ by letting $Y_i=1$ if the ball at $i$-th drawing is white and $Y_i=0$ if it is black. The sequence obtained by this procedure is exchangeable and satisfies
\begin{align*}
\text{Prob} & (Y_1=a_1,\ldots,Y_i=a_i) \\
& =\int_0^1u^j(1-u)^{i-j} \mu_t(du)=\int_0^1u^j(1-u)^{i-j} \frac{u^{t-1}(1-u)^{t-1}}{B(t,t)}du
\end{align*}
when $(a_1,a_2,\ldots,a_i) \in \{0,1\}^i$ and $\sum_{k=1}^i a_k= j$ (cf. \cite{F}). Here, the parameter $t$ represents the ratio of the initial number of white balls ($=$ that of black balls) to the number of balls added after each drawing. The case where no ball is added corresponds to $t=\infty$. 
\end{remark}

Finally, we give a sufficient and necessary condition for $\tilde{R}_{\nu}=A_{g_{\nu}} \in V_P(\infty)$ in terms of $\nu$.

\begin{theorem}
Let $\nu$ be a probability measure on $\N_0$ satisfying $\nu(i) >0$ for all $i \in \N_0$. Then, $\tilde{R}_{\nu} \in V_P(\infty)$ if and only if $\nu$ is a negative binomial distribution or a poisson distribution. Moreover, if so, (\ref{eq:conditionsup}) holds.
\end{theorem}
\begin{proof}
First, assume $\tilde{R}_{\nu}=A_{g_{\nu}} \in V_P(\infty)$. Then, by Theorem \ref{thm:gcondition}, $g_{\nu}$ must be $\a G_t $ or $\a Id$. If $g_{\nu}=\a G_t$, then $\nu(i)=\nu(0) \prod_{j=1}^{i} (\a G_t(j))^{-1} = \nu(0) (\a t)^{-i} \frac{(t)_i}{i!}$. Since $\nu$ is a probability measure, $\nu(0)$ must be $p^t$ where $p=1-\frac{1}{\a t}$ with $\a t>1$. Namely, $\nu$ is the negative binomial distribution with parameters $(t,1-\frac{1}{\a t})$ and in particular, if $t=1$, then $\nu$ is the geometric distribution with parameter $1-\frac{1}{\a}$. Also, if $g_{\nu}=\a Id$, then $\nu$ is the poisson distribution with parameter $\frac{1}{\a}$. 

On the other hand, if $\nu$ is the negative binomial distribution with parameters $(t,p)$, then $g_{\nu} =\frac{1}{(1-p)t} G_t$. In particular, if $\nu$ is the geometric distribution with parameter $p$ then $g_{\nu}=\frac{1}{1-p}G_1$. Also, if $\nu$ is the poisson distribution with parameter $\lambda$, then $g_{\nu}=\frac{1}{\lambda} Id$. Namely, for all the cases $\tilde{R}_{\nu}=A_{g_{\nu}} \in V_P(\infty)$ by Theorem \ref{thm:gcondition}.

Obviously, for all the cases, $\inf_{i \in \N} \{ g_{\nu}(i) \} >0$, so (\ref{eq:conditiontilder}) holds, and hence (\ref{eq:conditionsup}) holds.
\end{proof}

In this way, we found a sufficient condition for (\ref{eq:conditionsup}) in terms of $\nu$, but it is still restrictive, so to find a sufficient condition for more general $\nu$ (or $g$) is still an interesting open problem.

\section*{Acknowledgement}

We would like to thank Professor Takahiko Fujita and Professor Masato Takei for their valuable comments.

\end{document}